\documentclass[12pt,a4paper]{article}
\usepackage[utf8]{inputenc}
\usepackage{amssymb}
\usepackage{amsthm}
\usepackage{enumitem}

\usepackage{color}
\definecolor{DarkOlive}{rgb}{0.1047,0.2412,0.0064}
\definecolor{FireBrick}{rgb}{0.5812,0.0074,0.0083}
\definecolor{RoyalBlue}{rgb}{0.0236,0.0894,0.6179}
\definecolor{RoyalGreen}{rgb}{0.0236,0.6179,0.0894}
\definecolor{RoyalRed}{rgb}{0.6179,0.0236,0.0894}
\definecolor{LightBlue}{rgb}{0.8544,0.9511,1.0000}
\definecolor{Black}{rgb}{0.0,0.0,0.0}
\definecolor{MidnightBlue}{rgb}{0.0035,0.0020,0.1363}
\definecolor{FireBrick3}{rgb}{0.5812,0.0074,0.0083}
\definecolor{FireBrick4}{rgb}{0.2156,0.0023,0.0035}
\definecolor{Blue2}{rgb}{0.0000,0.0000,0.8644}
\definecolor{Navy}{rgb}{0.0000,0.0000,0.1927}
\definecolor{MediumBlue}{rgb}{0.0000,0.0000,0.6179}

\usepackage[
        bookmarks=false,
        pdftitle={Standard Generators of Finite Fields and their Cyclic Subgroups},
        colorlinks=true,backref=false,breaklinks=true,linkcolor=MediumBlue,
        citecolor=FireBrick3,filecolor=RoyalRed,breaklinks=true,
        urlcolor=Blue2]{hyperref}

\usepackage[hyperref,backend=bibtex,style=alphabetic,
            doi=false,url=false,isbn=false,giveninits=true]{biblatex}
\DeclareFieldFormat{postnote}{#1}
\DeclareFieldFormat{multipostnote}{#1}
\def\+{\discretionary{}{}{}}

\theoremstyle{plain}
\newtheorem{Thm}{Theorem}[section]

\newtheorem{Pro}[Thm]{Proposition}
\newtheorem{La}[Thm]{Lemma}

\theoremstyle{definition}
\newtheorem{Def}[Thm]{Definition}
\newtheorem{Rem}[Thm]{Remark}
\newtheorem{Alg}[Thm]{Algorithm}

\newenvironment{Prf}{\noindent\textbf{Proof.}}{\hfill $\Box$ \medskip}

\newcommand{\mod}{ \textrm{ mod } }
\newcommand{\F}{\mathbb{F}}
\newcommand{\Z}{\mathbb{Z}}
\newcommand{\C}{\mathbb{C}}
\newcommand{\Q}{\mathbb{Q}}

\newcommand{\lcm}{{\rm lcm}}
\newcommand{\GF}[1]{\F_{#1}}
\newcommand{\GFq}{\GF{q}}
\newcommand{\GFp}{\GF{p}}

\def\ceasy{(A) }
\def\ceasycode{(B) }
\def\conditer{(C) }
\def\ceff{(D) }

\begin{document}
\title{Standard Generators of Finite Fields and their Cyclic Subgroups}
\author{Frank Lübeck
\thanks{Funded by the Deutsche Forschungsgemeinschaft (DFG, German Research                                                                                         
Foundation) – Project-ID 286237555 – TRR 195}}
\maketitle
\begin{abstract}
We define  standardized constructions  of finite fields,  and standardized
generators of (multiplicative) cyclic subgroups in these fields.

The motivation is to provide a substitute for Conway polynomials which can
be used  by various software  packages and in collections  of mathematical
data which involve finite fields.

\end{abstract}

\section{Introduction}

For each prime number $p$ and $n \in \Z_{>0}$ there exists a unique finite
field $\GF{p^n}$  of order  $p^n$, up  to isomorphism.  A standard  way to
compute with such  a field is to specify an  irreducible polynomial $f \in
\GFp[X] \cong  \Z/p\Z[X]$ of degree $n$  and to use $\GF{p^n}  = \GFp[X] /
(f)$ where each element of the field is represented by a unique polynomial
of degree $< n$. Roughly $(1/n)$-th  of all polynomials of degree $n$ over
$\GFp$ are  irreducible, so there are  many ways to realize  $\GF{p^n}$ in
this way.

The first goal  of this paper is to define  a standardized construction of
all finite fields which fulfills a list of conditions:
\begin{itemize}
\item[(A)] 
it is easy to understand knowing the standard facts about finite fields,
\item[(B)]
it is easy to (re)-implement (say, given a basic polynomial arithmetic),
\item[(C)]
it  is iterative;  that  is the  construction  of a  new  field makes  use
of  previous constructions  of  proper subfields,  and  all subfields  are
naturally and effectively embedded in the new field,
\item[(D)]
it  is reasonably  efficient in  practice when  implemented with  straight
forward algorithms.
\end{itemize}

The  condition~\conditer means  that  we construct  the algebraic  closure
$\bar\GFp$ of  $\GFp$ by  constructing all  its finite  subfields together
with natural embeddings.

The range we have in mind in~\ceff is $n$ up to a few thousand for smaller
primes, larger for very small primes, and smaller for very large primes.

Our motivation  is to provide a  reference that could be  used in computer
algebra systems, other software packages or in collections of mathematical
data  which involve  finite fields.  A  unified description  of the  field
elements significantly simplifies the exchange and reuse of data.

Before getting  to the second  goal of this  paper let us  consider three 
previous approaches in this direction.                                    

In the classical article~\cite{Steinitz}  Steinitz described (in 1910) the
theory of field extensions as it  is taught nowadays in an algebra course.
And in  {\S 16} he gives  a very explicit construction  of $\bar\GFp$ (and
any  of its  subfields) which  fulfills our  conditions~\ceasy, \ceasycode
and~\conditer  above. Here  is  a sketch.  Steinitz  introduces a  natural
numbering of all polynomials over $\GFp$. For $m \in \Z_{>1}$ let $f_m \in
\GFp[X]$ be  the irreducible polynomial  of degree $m!$ with  the smallest
number. Then $K_m = \GFp[X]/(f_m)  \cong \GF{p^{m!}}$. It remains for $K_m
\leq K_{m+1}$ to define the embedding uniquely. Steinitz maps $X+(f_m)$ to
the  zero of  $f_m$ in  $K_{m+1}$ with  the smallest  number (elements  in
$K_{m+1}$ are  represented by  unique polynomials  of degree  $< (m+1)!$).
Obviously, this definition is not very practical, because the computations
of the polynomials and embeddings can only be done for very few small $m$;
and fields  of moderate size may  be only contained in  astronomically big
$K_m$. In  this article we  will extend Steinitz' definition  of numbering
and use it in places where certain choices need to be done.

As  second approach  to define  standardized models  for finite  fields we
mention the  work of Lenstra  and de Smit~\cite{LenstraSmit,  Handbook}. A
main goal  for them was a  variant of our condition~\ceff,  namely to give
a  description  with  good,  polynomial time,  asymptotic  behaviour,  but
the  emphasis  was not  on  practical  implementation. Their  construction
fulfills~\conditer and yields a natural $\GFp$-basis for each finite field
which  contains  the  corresponding  bases of  all subfields  as  subsets,
this  defines  natural  embeddings. Our construction  will  also have this
property.  Understanding  and  implementing  their  construction  needs  a
background in  algorithmic   number theory (computations  in number fields
of characteristic zero).

Finally we  mention the approach  given by Conway  polynomials. Originally
defined by  Richard Parker  these are  currently used  and available  in a
number of  computer algebra  systems with good  support for  finite fields
like   {GAP}~\cite{GAP},   {SageMath}~\cite{SAGE},   {MAGMA}~\cite{MAGMA},
{Macaulay2}~\cite{M2} and  various more specialized programs,  for example
the  {C-MeatAxe}~\cite{MeatAxe}.  The  Conway polynomial  $C_{p,n}(X)  \in
\GFp[X]$  is  a  monic  irreducible  polynomial of  degree  $n$  which  is
primitive and respects a certain compatibility with the Conway polynomials
which  define proper  subfields. Primitive  means that  the residue  class
of  $X$  in   $\GFp[X]  /  (C_{p,n}(X))  \cong   \GF{p^n}$  generates  the
multiplicative group  of the field,  that is it  is of order  $p^n-1$. The
compatibility means that  for any divisor $m$ of $n$  the residue class of
$X^{(p^n-1)/(p^m-1)}$ is a  zero of the Conway  polynomial $C_{p,m}$. This
also  defines embeddings of  the subfields.  There are  many sets  of
polynomials  fulfilling these  conditions. To  get a  well defined  set of
polynomials  there is  a further  (recursive) condition,  namely $C_{p,n}$
must be the smallest polynomial with the mentioned properties with respect
to some ordering of polynomials (which we do not define here). We refer to
the introduction of the Modular Atlas~\cite{ModAtl} for more details.

The   construction   of   Conway    polynomials   somehow   fulfills   our
conditions \ceasy  (one has  to  show the  existence) and~\ceasycode  (one
needs to compute  roots). Condition~\conditer is fine  for the embeddings,
but the  constructions of  subfields give  additional constraints  for the
next polynomial.  Unfortunately, condition~\ceff is a  problem here. There
are two basic methods to compute a new Conway polynomial: either enumerate
all monic  polynomials of the  right degree  and check the  conditions, or
enumerate all  compatible and primitive  polynomials to find  the smallest
one.  Even for  moderate parameters  both  enumerations can  be very  time
consuming. All  systems mentioned above  use a list of  precomputed Conway
polynomials~\cite{ConwayData}  whose generation  took  many  years of  CPU
time. It  is almost  impossible to compute  any further  Conway polynomial
$C_{p,n}$ when $n>1$ is not prime.

And there is another fundamental  problem: Primitivity can only be checked
if all  prime factors  of the  order $p^n-1$  of the  multiplicative group
are  known. These  factors  are  known in  many  cases  thanks to  decades
long enormous  computational efforts~\cite{BrentFactors}, but not  for the
majority of fields we would like to cover in practice.

A  motivation for  the definition  of  Conway polynomials  comes from  the
following fact: There exist group isomorphisms, which we will call a lift,
from  the  multiplicative  group  $\bar\GFp^\times$  to  the  subgroup  of
$\C^\times$ consisting of all roots of  unity whose order is not divisible
by $p$.  A well defined such  lift is explicitly determined  by the Conway
polynomials, its restriction to $\GF{p^n}  = \GFp[X] / (C_{p,n})$ is given
by $X + (C_{p,n}) \mapsto \exp(\frac{2 \pi i}{p^n-1})$.

This explicit lift is for example often used in the modular representation
theory  of  finite  groups  where  the  definition  of  Brauer  characters
depends  on  such a  lift.  There  exists  a  large collection  of  highly
non-trivial  representation  theoretical  data   in  the  character  table
library~\cite{CTblLib},  which includes  all the  data from  the Atlas  of
Brauer characters~\cite{ModAtl}. These data are stored with respect to the
lift defined by Conway polynomials.

The  inverse of  a lift  is  also used  in  this context,  namely for  the
reduction  of characters  in characteristic~0  modulo $p$.  The choice  of
a  lift  is  equivalent  to  the  choice  of  a  $p$-modular  system.  The
Atlas~\cite[Appendix 1]{ModAtl} contains tables which describe this map on
common  irrational numbers  with respect  to  the lift  defined by  Conway
polynomials.

Another application we are interested in comes from Deligne-Lusztig theory
where elements in  a torus over a finite field  are interpreted as complex
characters of a dual torus via a lift, see~\cite[3.1]{Ca93}.

Now we can  describe the second goal  of this paper: We want  to specify a
well defined lift $\bar\GFp^\times \to  \C^\times$ for the elements in our
standardized finite subfields. We do this  by defining for $m \in \Z_{>0}$
with $\gcd(m, p) = 1$ the element $y_m$ in a finite subfield of $\bar\GFp$
which is mapped  to $\exp(\frac{2 \pi i}{m})$. Our  definition will enable
us to compute $y_m$ in practice whenever  we know the prime factors of $m$
and we can construct  a field of order $p^n$ which  contains an element of
order $m$. (Note that  $p^n-1$ may be much larger than $m$  and that we do
not need to know all prime divisors of $p^n-1$.)

\textbf{Content of this article.}
In Section 2 we recall some  basic facts about finite fields. In Section~3
we define towers of finite fields and  explain how to use them to describe
an algebraic closure of $\GFp$. We will extend Steinitz' idea to enumerate
polynomials  and  finite  field  elements.  In  Section~4  we  explain  in
more  detail how  we describe  a lift  $\bar\GFp^\times \to  \C^\times$ by
specifying  appropriate  elements in  our  standardized  fields. The  core
of  the  paper is  in  Section~5,  where  we define  explicit  irreducible
polynomials of prime  degree on which the setup in  Section~3 depends, and
in  Section~6, which  contains the  explicit construction  of standardized
elements  of given  order  (which define  a lift).  We  have also  written
a  software  package~\cite{StandardFF},  based  on  GAP~\cite{GAP},  which
implements the  constructions in this  paper and  which we used  to verify
(and  improve)  the practicality  of  our  descriptions. In  Section~7  we
collect some remarks concerning this implementation. Finally, in Section~8
we discuss  the question  of translating the  values of  Brauer characters
from one lift to another one.

\textbf{Complexity  considerations.} There  is a  lot of  literature which
is  relevant  in   the  context  of  this  article,   e.g.,  on  efficient
arithmetic  in  field  extensions, irreducibility  tests,  computation  of
minimal polynomials,  construction of irreducible  polynomials, embeddings
of   fields.   While  working   on   this   article  and   the   reference
implementation~\cite{StandardFF} we got  the impression that sophisticated
algorithms with good  asymptotic complexity do not  give vast improvements
in the  range we want  to consider  in practice, say  degrees up to  a few
thousands.  Therefore  we  do  not  include  statements  about  asymptotic
complexity  here, but  just mention  what works  sufficiently well  in our
straight forward implementation.

We  will construct  our fields  $\GF{p^n}$ by  their defining  irreducible
polynomial together  with a base change  matrix of size $n  \times n$ over
$\GFp$  (this  amounts  to  storing  $n$  elements  of  the  field)  which
reduce the  computations of embeddings  of elements into larger  fields to
matrix-vector multiplications.

It  should be  possible to  use  our standardized  fields within  existing
schemes  for  computing  in  compatible lattices  of  finite  fields,  see
for  example those  described  in~\cite{BCS},  in~\cite{DDSlTowers} or  in
\cite{DRR}.

\textbf{Acknowledgements.} First I wish to  thank Hendrik Lenstra for very
interesting  discussions about  the standardization  of finite  fields. In
particular Lenstra convinced  me \emph{not} to publish an  article which I
had prepared some years ago. In that  draft I proposed a variant of Conway
polynomials which  can be computed  in practice for all  fields $\GF{p^n}$
for which the factorization of $p^n-1$ is  known. But it is much better to
separate the  construction of  the fields and  the definition  of standard
generators of cyclic subgroups, as we do in the present article.

Furthermore  I   thank  Wilhelm  Plesken   for  sending  me   his  lecture
notes \cite{Pl15} and  for his  permission to freely  reuse his  ideas for
this article  (e.g., our definition of  Steinitz numbers, and a  sketch of
our Section~\ref{secdefcyc} can be found in these notes).

I  thank  Thomas  Breuer for  useful  discussions  about  Brauer
character  values,  and for helpful  comments on earlier versions  of this
paper and the related GAP package.

Finally, I thank Richard Parker and Claus Fieker for their comments on an
earlier version of this paper.

\section{Notation and basic facts about finite fields}\label{secnotation}

Let $p$ be a prime, we will use $q$, $q'$ for powers of $p$.

We start  with recalling some  elementary facts about finite  fields which
can be  found in  many algebra  text books,  e.g.,~\cite[V.5]{Lang}. These
will be used in the sequel without further reference.

\begin{Rem}\label{remgeneral}
\begin{enumerate}[label=(\alph*)]
\item For every prime power $q$ there exists up to isomorphism exactly one
finite field  $\GFq$ with $q$ elements.  It is the splitting  field of the
polynomial $X^q-X$ over its prime field $\GFp$.
\item Let $q'  = p^a$, $q =  p^b$. The field $\GF{q'}$ is  isomorphic to a
subfield of $\GFq$ if and only if $q$  is a power of $q'$, that is $a \mid
b$. In that case this subfield is unique and consists of the zeroes of the
polynomial $X^{q'}-X$.
\item  Let $\GF{q'}  \leq \GFq$  be a  subfield. This  field extension  is
Galois, the Galois group is cyclic  and generated by the map $\sigma_{q'}:
\GFq \to \GFq, x \mapsto x^{q'}$ (and of order $d$ if $q = (q')^d)$.
\item The multiplicative group $\GFq^\times$ is cyclic of order $q-1$.
\item The field $\GFq$ is perfect, that is every irreducible polynomial in
$\GFq[X]$ has pairwise distinct roots.
\end{enumerate}
\end{Rem}

We use the following terminology.
\begin{Def}\label{primelt}
Let $\GFq$ be a finite field with prime field $\GFp$.
\begin{enumerate}[label=(\alph*)]
\item We  call an element  $x \in \GFq$  a \emph{primitive element}  if it
generates the field over its prime field, that is $\GFq = \GFp[x]$.
\item  We call  an element  $x  \in \GFq$  a \emph{primitive  root} if  it
generates the multiplicative group $\GFq^\times$.
\end{enumerate}
\end{Def}

\begin{Rem}\label{primtensor}
\mbox{}
\begin{enumerate}[label=(\alph*)]
\item All finite fields  have a primitive root, and a  primitive root is a
primitive element.
\item Let  $m,n \in \mathbb{Z}_{>1}$  with $\gcd(m,n) =  1$ and let  $K, L
\leq \bar\GFq$ (an algebraic closure of $\GFq$) be algebraic extensions of
$\GFq$ of degree $m$ and $n$ with primitive elements $x, y$, respectively.
Then $KL := K[y]  = \GFq[x][y] = \GFq[y][x] = L[x]$ is  of degree $mn$ and
$xy$ is a primitive element, that is $KL = \GFq[xy]$.
\end{enumerate}
\end{Rem}
\begin{proof}
Assume that  $xy$ is  not a generator.  Then it is  contained in  a proper
maximal subfield $F \leq KL$ of prime index. Since $\gcd(m,n) = 1$ we have
$K \leq F$ or $L  \leq F$ and so $x \in F$ or $y \in  F$. But with $xy \in
F$ we get that both, $x \in F$ and $y \in F$, a contradiction.

Remark: the same argument shows that  $x+y$ is also a primitive element of
$KL$.
\end{proof}

We describe algebraic field  extensions via irreducible polynomials. These
are considered in the following lemmas.

\begin{La}\label{numirrpols}
Let $q$ be a prime power and $r$ be a prime.
\begin{enumerate}[label=(\alph*)]
\item There exist $(q^r-q)/r$ monic  irreducible polynomials of degree $r$
in $\GFq[X]$.
\item Assume that $r \nmid (q-1)$.  Then for any $c \in \GFq^\times$ there
are  $(q^r-q)/(r(q-1))$ monic  irreducible  polynomials of  degree $r$  in
$\GFq[X]$ with constant term $c$.
\end{enumerate}
\end{La}
\begin{proof}
(a)  Each monic  irreducible  polynomial  $f \in  \GFq[X]$  of degree  $r$
generates the  field $\GF{q^r} \cong \GFq[X]  / (f)$. Since $r$  is prime,
all $q^r-q$ elements of $\GF{q^r} \setminus \GFq$ generate $\GF{q^r}$ over
$\GFq$. So,  their minimal  polynomials have degree  $r$ and  $r$ distinct
roots (more  precisely, for  $x \in  \GF{q^r} \setminus  \GFq$ the  set of
conjugates $\{x,  \sigma_q(x), \ldots,  \sigma_q^{r-1}(x)\}$ has  size $r$
and these  all have  the same minimal  polynomial $\prod_{i=0}^{r-1}  (X -
\sigma_q^i(x))$).

(b)  The  norm  map   $N_{q^r/q}:  \GF{q^r}^\times  \to  \GFq^\times$,  $x
\mapsto \prod_{i=0}^{r-1} \sigma_q^i(x) = x^{1+q+q^2+\ldots+q^{r-1}}$ is a
surjective homomorphism  (the image  of a primitive  root has  order $q-1$
because $q^r-1 = (q-1)  (1+q+q^2+\ldots+q^{r-1})$). The restriction of the
norm  map to  $\GFq^\times$  is  $x \mapsto  x^r$,  hence an  automorphism
because $r \nmid (q-1)$. Therefore, every $c \in \GFq^\times$ has the same
number of preimages under the norm  map in $\GF{q^r} \setminus \GFq$. This
shows~(b) because  the constant term of  the minimal polynomial of  $x \in
\GF{q^r} \setminus \GFq$ is $(-1)^r N_{q^r/q}(x)$.
\end{proof}

\begin{La}\label{degrirred} 
\mbox{}
\begin{enumerate}[label=(\alph*)]
\item
Let $K$ be a field of characteristic $p$. For any $a \in K$ the polynomial
$X^p - X -a \in K[X]$ either has a root in $K$ or it is irreducible.
\item
Let  $r$ be  a  prime and  let $K$  be  a field.  For  any $a  \in K$  the
polynomial  $X^r  - a  \in  K[X]$  either  has a  zero  in  $K$ or  it  is
irreducible.
\end{enumerate}
\end{La}
\begin{proof}
(a) (Artin-Schreier extensions) Let $b$ be a zero of the polynomial $X^p -
X -a$ in a  splitting field of this polynomial over  $K$. Since $\GFp \leq
K$ (the zeroes of $X^p-X$) we have $X^p  -X -a = \prod_{i \in \GFp} (X - b
-  i)$. The  minimal  polynomial of  $b$  over $K$  is  a partial  product
$\prod_{i \in I}  (X - b - i)  \in K[X]$, with $I \subseteq  \GFp$, say of
size $k$. Then  the coefficient of $X^{k-1}$, which is  of the form $kb+j$
with $j \in \GFp$, shows that $k = p$ or $b \in K$ and so $k=1$.

(b) Let  $b_1$ be  a zero of  $X^r -a$  in a splitting  field $L$  of this
polynomial over $K$. Let  $b_2, \ldots, b_k \in L$ be  the other zeroes of
the minimal polynomial of $b_1$ over $K$.  Then we have for $b' := b_1 b_2
\cdots b_k \in K$ that $(b')^r = b_1^r  b_2^r \cdots b_k^r = a^k$. If $k <
r$ then $k$  is prime to $r$  and there exist $l,  k' \in \mathbb{Z}_{>0}$
with $kl = 1+k' r$. So $(b')^{rl} = a^{kl} = a \cdot a^{k' r}$. This shows
that $a$ has an $r$-th root in $K$.
\end{proof}

The following observation will be useful for finding elements which do not
have an $r$-th root in finite fields.  For a prime $r$ and integers $t \in
\mathbb{Z}_{\geq  0}$, $m  \in \mathbb{Z}_{>0}$  we write  $r^t ||  m$, if
$r^t$ divides $m$ but $r^{t+1}$ does not divide $m$.

\begin{La}\label{rparts}
Let $q$ be a power of a prime $p$ and $r \neq p$ another prime.
\begin{enumerate}[label=(\alph*)]
\item The smallest  power $n$ such that  $q^n - 1$ is divisible  by $r$ is
the order of $q$ modulo $r$. It is a divisor of $r-1$.
\item Let $r$ be an odd divisor of  $q-1$, say $r^t || (q-1)$. For any $n$
with $r^i || n$ we have $r^{t+i} || (q^n-1)$.
\item Let $r=2$  (and so $q$ odd)  and $2^t || (q^2-1)$. For  any $n$ with
$2^i || n$ we have $2^{t+i} || (q^{2n}-1)$.
\end{enumerate}
\end{La}
\begin{proof}
(a) We  have $r  | (q^n-1)$  if and  only if  $q^n \equiv  1 \mod  r$; and
$(\mathbb{Z}/ r \mathbb{Z})^\times$ is cyclic of order $r-1$.

(b) Write  $q^n-1 =  (q-1)(1+q+q^2+ \ldots +q^{n-1})$.  Since $q  \equiv 1
\mod r$ we see that the second factor is $\equiv n \mod r$. This shows the
case $i = 0$, that is $ r \nmid  n$. We now assume that $n=r$ and write $q
= 1+rs$. Then  $q^j = 1 + j(rs)  + c_j \cdot r^2$ for  some integer $c_j$.
And so  $1+q+q^2+ \ldots +q^{r-1}  \equiv r + \frac{r(r-1)}{2}rs  \mod r^2
\equiv r  \mod r^2$.  This shows  $r ||  (1+q+q^2+ \ldots  +q^{r-1})$. The
general case follows by induction.

(c) The argument for odd $n$ is the same as in~(b). The case $n=2$, $q^4-1
= (q^2-1)(q^2+1)$ is clear because $q^2 \equiv 1 \mod 4$. The general case
follows by induction.
\end{proof}

In the end of this section we define a function which we will use later in
several places.  For a fixed  integer $q$ it  provides a bijection  on the
range of integers $i$  with $0 \leq i < q$ which appears  to behave like a
random number generator in our use cases.

\begin{Def}\label{SAS}
Let $q > 0$ and $i$ be integers. We call 
\[\texttt{StandardAffineShift}(q, i) \]
the integer $(m i + a) \textrm{ mod } q$, where
$m$ is the largest integer with $m \leq \frac{4}{5}q$ and $\gcd(m,q) = 1$
and $a$ is the largest integer $\leq \frac{2}{3}q$.
\end{Def}

\section{Towers of finite fields}\label{sectowers}

\begin{Def}\label{deftower}
Let  $F$  be  a  field   and  $\{X_i\mid  i=1,\ldots,l\}$  be  independent
commuting   indeterminates   over   $F$.   For  $i   =   1,\ldots,l$   let
$f_i   \in   F[X_1,\ldots,X_{i-1}][X_i]$   be  monic   in   the   variable
$X_i$.   We   assume   that   for   $i=1,\ldots,l$   the   residue   class
of    $f_i$    in    
\[F[X_1,\ldots,X_{i-1}][X_i]/(f_1,\ldots,f_{i-1})    =
(F[X_1,\ldots,X_{i-1}]/(f_1,\ldots,f_{i-1}))[X_i]\] 
is irreducible.

Then the  sequence $((X_i, f_i),  i=1,\ldots,l)$ defines a  \emph{tower of
(algebraic) field  extensions over $F$},  that is $F  = F_0 \leq  F_1 \leq
\ldots \leq F_l$ where $F_i = F[X_1,\ldots,X_{i}]/(f_1,\ldots,f_{i})$.

The degree $d_i = [F_i:F_{i-1}]$ is  the degree of the polynomial $f_i$ in
the indeterminate $X_i$.
\end{Def}

\begin{Rem}\label{remcomputetower}
Let $((X_i,  f_i), i=1,\ldots,l)$  be a  tower of  field extensions  for a
sequence  of  fields $F  =  F_0  \leq F_1  \leq  \ldots  \leq F_l$  as  in
Definition~\ref{deftower}.  Then  the  residue  class  of  any  polynomial
$\tilde g \in  F[X_1,\ldots,X_l]$ modulo $(f_1,\ldots, f_l)$  has a unique
representative $g \in  F[X_1,\ldots,X_l]$ where the degree of  $g$ in each
variable $X_i$ is smaller than $d_i$.

This  representative $g$  can be  constructed recursively:  First consider
$\tilde g$ as polynomial in $X_l$ and reduce it using the monic polynomial
$f_l$ until the degree of $\tilde g$  in $X_l$ is smaller than $d_l$. Then
proceed  in the  same  way with  $X_{l-1}, \ldots,  X_1$.  Since $f_i  \in
F[X_1,\ldots,X_i]$, the reductions of powers of $X_i$ will not enlarge the
degree in the previously considered variables $X_j$, $j>i$. (The $\{f_i\}$
form  a Gröbner  basis of  the ideal  they generate  with respect  to the
reverse  lexicographic  monomial  ordering,  and  we  just  described  the
standard reduction with this Gröbner basis.)
\end{Rem}

\begin{Def}[Tower basis] \label{DefTowerBasis}
Let $((X_i,  f_i), i=1,\ldots,l)$  be a  tower of  field extensions  for a
sequence of fields $F = F_0 \leq  F_1 \leq \ldots \leq F_l$. We define the
\emph{tower basis} of each $F_i$, $0 \leq  i \leq l$ recursively, it is an
ordered $F$-basis whose elements are  represented by the reduced monomials
in $\{X_1, \ldots, X_i\}$.

For $i=0$, $F_0 =  F$, the basis is $(1)$. Let $i >  0$ and $(b_0, \ldots,
b_m)$  be the  already  defined basis  for $F_{i-1}$.  Then we define  the
concatenation of $(b_0  X_i^j, b_1 X_i^j, \ldots, b_m X_i^j)$  for $j = 0,
1, \ldots, d_i-1$ as representatives of the tower basis of $F_i$ (where as
before $d_i$ is the degree of $f_i$ in $X_i$).
\end{Def}

Following Plesken~\cite{Pl15} we now define  a numbering of field elements
in towers over  a finite prime field $\GFp$. This  extends a definition of
Steinitz~\cite[\S 16]{Steinitz}.

\begin{Def}[Steinitz number]\label{DefSteinitz}
Let $p$ be a prime and let  $((X_i, f_i), i=1,\ldots,l)$ define a tower of
field extensions  over $\GFp = F_0  \leq \ldots \leq F_l$,  where $d_i$ is
the  degree of  $F_i$ over  $F_{i-1}$ and  $a_i =  \prod_{j=1}^i d_j$  the
degree of $F_i$ over $F_0$.

We define an injective map $s: F_l \to \Z$, such that $s(F_i) = \{m \in \Z
\mid\; 0 \leq m \leq |F_i|-1 = p^{a_i} - 1\}$ for all $i$. For $x \in F_l$
we call $s(x)$ the Steinitz number of $x$.

If $i=0$  we identify  $\GFp =  \Z/p\Z$ and define  $s(x) =  k$ when  $x =
k+p\Z$ with $0 \leq  k < p$. For $i>0$ assume that  $s$ is already defined
on  $F_{i-1}$.  Each  $x  \in  F_i  =  F_{i-1}[X_i]/(f_i)$  has  a  unique
representative $g  = c_0 +  c_1 X_i +  \ldots + c_{d_i-1}  X_i^{d_i-1} \in
F_{i-1}[X_i]$  of degree  $< d_i$.  We define  $s(x) =  \sum_{j=0}^{d_i-1}
s(c_j) q_{i-1}^j$ where $q_{i-1} = |F_{i-1}| = p^{a_{i-1}}$.

We  also  define  the  Steinitz  number   $s(f)$  of  a  polynomial  $f  =
\sum_{j=0}^k c_j  X^j \in F_l[X]$ using  the Steinitz numbers on  $F_l$ by
$s(f) = \sum_{j=0}^k s(c_j) q_l^j$.
\end{Def}

Using  Remark~\ref{remcomputetower} it  is  easy to  compute the  Steinitz
number of  an element in  the tower of  field extensions represented  by a
polynomial in $\GFp[X_i, i=1,\ldots,l]$. And  vice versa, given a Steinitz
number $m$, it is easy to write down a polynomial representing the element
$x$ with $s(x) = m$ by computing the $q_l$-adic decomposition of $m$, then
the $q_{l-1}$-adic decomposition of the coefficients and so on.

Also note  the connection to the  tower basis $(b_0, \ldots,  b_{n-1})$ of
$F_l$ defined in~\ref{DefTowerBasis}: Let $x \in F_l$ with Steinitz number
$s(x) = m$.  Consider the $p$-adic expansion $m  = m_0 + m_1 p  + \ldots +
m_{n-1}p^{n-1}$,  where  $0  \leq  m_i  <  p$  for  all  $i$.  Then  $x  =
\sum_{i=0}^{n-1} (m_i \mod p) b_i$. (So, the  $p$-adic  expansion  of  $m$
yields the coefficients of $x$ with respect to the tower basis.)
\bigskip

Let  $p$ be  a prime,  $n \in  \mathbb{Z}_{>0}$ and  $n =  r_1^{l_1}\cdots
r_k^{l_k}$ be the prime factorization of $n$ with $r_1 < \ldots < r_k$.

We want to  describe and construct the finite field  $\GF{p^n}$. Since for
every  divisor $m  \mid  n$  there is  a  unique  subfield $\GF{p^m}  \leq
\GF{p^n}$  there  exists  a  unique sequence  of  field  extensions  $\GFp
\leq  \GF{p^{r_1}}  \leq  \ldots  \leq \GF{p^{n/r_k}}  \leq  \GF{p^n}$  of
non-decreasing prime degrees.

Let $q_i := p^{(r_i^{l_i})}$, $i  = 1,\ldots,k$. Then the field extensions
$\GFp \leq \GF{q_i}$ are of  prime power degree $r_i^{l_i}$ and $\GF{p^n}$
is the compositum  $\GF{p^n} = \GF{q_1}\cdots \GF{q_k}$,  where any factor
only intersects with the product of the others in the prime field $\GF{p}$
(see  Remark~\ref{primtensor}(b)). This  shows that  $\GF{p^n}$ can  be
constructed using extensions of prime power degree $r^l$ of $\GFp$.

So let $r$ be  a prime (equal to $p$ or not)  and $l \in \mathbb{Z}_{>0}$.
We will construct the field  $\GF{p^{(r^l)}}$ via a sequence of extensions
of degree $r$:
\[\GFp \leq \GF{p^r} = \GFp[x_{r,1}] \leq \ldots \leq \GF{p^{(r^l)}} =
\GF{p^{(r^{l-1})}}[x_{r,l}]. \]
For   this   we   need   to  construct   recursively   monic   irreducible
polynomials  
\[\bar{f}_{r,i}(X_{r,i})   \in  \GF{p^{(r^{i-1})}}[X_{r,i}]  =
\GFp[x_{r,1},\ldots, x_{r,i-1}][X_{r,i}]\] 
of    degree    $r$    (for    $i=1,\ldots,    l$)    where    we    write
$x_{r,i}$     for     the     residue    class     of     $X_{r,i}$     in
$\GFp[x_{r,1},\ldots,x_{r,i-1}][X_{r,i}]/(\bar{f}_{r,i})$.  It   is  clear
that $x_{r,i}$ is of degree $r^i$ over the prime field $\GFp$.

An  $\GFp$-basis  of   $\GFp[x_{r,1},\ldots,x_{r,i-1}]$  consists  of  the
elements  
\[\{x_{r,1}^{j_1}\cdots  x_{r,i-1}^{j_{i-1}}  \mid  0  \leq  j_1,
\ldots,  j_{i-1} \leq  r-1\}.\] 
Changing    such     basis    elements    to     representing    monomials
$X_{r,1}^{j_1}\cdots   X_{r,i-1}^{j_{i-1}}$   in   the   coefficients   of
$\bar{f}_{r,i}$  we get  a  polynomial  $f_{r,i} \in  \GFp[X_{r,1},\ldots,
X_{r,i}]$.

Then   we   have    $\GF{p^{(r^i)}}   =   \GFp[X_{r,1},\ldots,X_{r,i}]   /
(f_{r,1},\ldots,f_{r,i})$, that is the  sequence \[ ((X_{r,i}, f_{r,i}), i
= 1, \ldots, l), \] defines a  tower of field extensions over $\GFp$, each
of degree $d_i = r$, as in Definition~\ref{deftower}.

\subsection{Construction of an algebraic closure $\bar\GFp$}\label{secbarFp}

If we define polynomials $f_{r,i}$ as above  for all primes $r$ and $i \in
\Z_{>0}$ we get an explicit description of an algebraic closure $\bar\GFp$
of $\GFp$, because each element $\bar f \in \bar\GFp$ is contained in some
finite subfield $\GF{p^n}$.

\begin{Rem}\label{rembargfp}
This construction has a number of nice properties:
\begin{itemize}
\item[(a)]
Each $\bar f  \in \GF{p^n} \subset \bar \GFp$ has  a unique polynomial \[f
\in  \GFp[X_{r,i}\mid\;  r  \textrm{  prime,  }i  \in  \Z_{>0},  r^i  \mid
n],\]  which has  degree $<  r$ in  each variable  $X_{r,i}$, as  standard
representative.
\item[(b)] The representation in~(a) does  not depend on $n$. The smallest
possible $n$ has $r$-part $r^i$ if $X_{r,i}$ occurs in a non-zero monomial
of $f$, but not $X_{r,j}$ with $j > i$.
\item[(c)]  Each  element  $x  \in  \bar\GFp$ can  be  identified  by  its
\emph{Steinitz  pair}  $(n, m)$  where  $n$  is  the  degree of  $x$  over
$\GFp$ and  $m$ is  the Steinitz  number of $x$  as element  of $\GF{p^n}$
(see~\ref{DefSteinitz}, note  that we  use the tower  which has  the prime
divisors of $n$ in non-decreasing order as relative degrees).
\item[(d)]  In  particular,  the  representatives  in~(a)  yield  explicit
natural embeddings  $\GF{p^m} \hookrightarrow  \GF{p^n}$ whenever  $m \mid
n$. In that case the monomials representing the tower basis of  $\GF{p^m}$
are a subset of the monomials representing the tower basis of $\GF{p^n}$.
\item[(e)] If $n  = r_1^{l_1} \cdots r_k^{l_k}$, then  $x_n := x_{r_1,l_1}
\cdots  x_{r_k,l_k}$ is  a primitive  element of  $\GF{p^n} =  \GFp[x_n]$,
see~\ref{primtensor}(b).
\item[(f)] We can  perform arithmetic in $\bar \GFp$: Let  $\bar f, \bar g
\in \bar\GFp$ with  standard representatives $f, g$ as  in~(a). Then $f\pm
g$ is  the standard  representative of  $\bar f  \pm \bar  g$. We  get the
standard representative of $\bar f \bar g$  from $f g$ by reducing it with
the  $f_{r,  i}$, starting  with  the  lexicographically largest  $(r,i)$,
see~\ref{remcomputetower}.

If  $\bar f  \in \GF{p^n}^\times$  its inverse  can be  computed as  $\bar
f^{-1}  = \bar  f^{p^n-2}$ (via  repeated  squaring). In  large fields  it
is  more  efficient  to  use  the extended  Euclidean  algorithm  for  the
representative  $f$ and  $f_{r,i}$, when  $X_{r,i}$ is  the variable  with
lexicographically  largest  $(r,i)$  occuring  in $f$  (this  may  involve
further inversions in the coefficient field).
\end{itemize}
\end{Rem}

\section{Embedding $\bar\GFp^\times$ into $\mathbb{C}^\times$}
\label{secembedding}

\begin{Pro}
Let $p$ be a  prime. Let $\bar\GFp$ be an algebraic  closure of the finite
prime  field $\GFp$  and $\bar\GFp^\times$  its multiplicative  group. Let
$\Q_{p'}$ be the  additive group of rational numbers  whose denominator is
not divisible by  $p$, the additive group of $\Z$  is a subgroup. Finally,
let $\mu_{p'}  \leq \C^\times$ be the  subgroup of complex roots  of unity
whose order is not divisible by $p$.

\begin{itemize}
\item[(a)] The  exponential map 
\[ \Q^+ \to  \C^\times, \quad \frac{r}{s} \mapsto  e^{2\pi i \frac{r}{s}},
\]
induces an isomorphism $e: \Q_{p'}/\Z \to \mu_{p'}$.
\item[(b)]
There exists an isomorphism
\[ \ell: \bar\GFp^\times \to \Q_{p'}/\Z.\]
\end{itemize}
\end{Pro}

\begin{Prf}
Part~(a) is clear.

For part~(b) we show the existence of  such a map by induction. Let $K_m =
\GF{p^{m!}}$, then  $K_m \leq K_k$ if  $m \leq k$ and  $\bar\GFp = \cup_{m
\in  \Z_{>0}} K_m$.  For $m=1$  the multiplicative  group $K_1^\times$  is
cyclic of order  $p-1$ and generated by a primitive  root $x_1$. We define
$\ell$ on $K_1^\times$  by $x_1 \mapsto 1/(p-1) \mod \Z$.  Now assume that
$\ell$ is defined on $K_m^\times$ by mapping a primitive root $x_m \mapsto
1/(p^{m!}-1)$.  Let $y$  be  a primitive  root of  $K_{m+1}$.  Then $y'  =
y^{(p^{(m+1)!}-1)/(p^{m!}-1)}$  is a  primitive  root of  $K_m$ and  there
exists a  $k \in \Z$  with $(y')^k  = x_m$. Then  set $x_{m+1} =  y^k$ and
define  $\ell$ on  $K_{m+1}^\times$ by  $x_{m+1} \mapsto  1/(p^{(m+1)!}-1)
\mod \Z$, this extends the map previously defined on $K_m$.

The injectivity of $\ell$ is clear by  construction. Let $a \in \Z$ be not
divisible by $p$, then $p$ is prime to  $a$ and there is a $j \in \Z_{>0}$
with $p^j \equiv  1 \mod a$, that  is $a \mid (p^j -  1) \mid (p^{j!}-1)$.
This shows  that $1/a \mod  \Z$ is  in the image  of $\ell$, so  $\ell$ is
surjective.
\end{Prf}

We are  interested in  an explicit  computable description  of such  a map
$\ell$  and  so  the  induced  lift $e  \circ  \ell:  \bar\GFp^\times  \to
\C^\times$ as  in the proposition in  terms of an explicit  description of
$\bar\GFp$.

\begin{Rem}\label{defcyclic}
The cyclic group  $\GF{p^n}^\times$ is the direct product  of its (cyclic)
Sylow subgroups. A homomorphism  $\ell: \GF{p^n}^\times \to \Q_{p'}/\Z$ is
uniquely  determined by  specifying  an arbitrary  generator $y_{n,r}$  of
the  Sylow $r$-subgroup  of $\GF{p^n}^\times$  such that  $\ell(y_{n,r}) =
\frac{1}{r^t} \mod \Z$ for each prime $r$ with $r^t || (p^n-1)$.

Having fixed  these $y_{n,r}$ it is  easy to compute for  each divisor $m$
of  $p^n-1$  the  element  $y_m \in  \GF{p^n}^\times$  with  $\ell(y_m)  =
\frac{1}{m}$, provided the prime factorization of $m$ is known.

For example, for $m = p^n-1 = \prod_r r^{t_r}$ the element $y'_m = \prod_r
y_{n,r}$ (product over all prime divisors  $r$ of $m$) is a primitive root
with $\ell(y'_m) = \sum_r \frac{1}{r^{t_r}} \mod \Z = \frac{a}{m} \mod \Z$
for some $a  \in \Z$. Let $b \in  \Z$ be the inverse of  $a \mod (p^n-1)$.
Then $y_m :=  (y'_m)^b$ is the element with $\ell(y_m)  = \frac{1}{m} \mod
\Z$.

Note  that in  the  analogous  construction for  arbitrary  divisors $m  |
(p^n-1)$ only  appropriate powers of  $y_{r,n}$ for prime divisors  $r$ of
$m$ are needed.
\end{Rem}

\section{Definition of standard extensions of prime degree}
\label{secstdgf}

As before let $p$ be a fixed prime. We now define polynomials $f_{r,i} \in
\GFp[X_{r,j}\mid\;  1 \leq  j  \leq i]$  for  each prime  $r$  and $i  \in
\Z_{>0}$  as  explained  in Section~\ref{secbarFp}.  We  distinguish  four
cases:
\begin{itemize}
\item $r = p$, 
\item  $r \mid (p-1)$ and in case $r = 2$ also $4 \mid (p-1)$,
\item $r = 2$ and $4 \mid (p+1)$
\item other $r$ (that is $r \neq 2, p$ and $r \nmid (p-1)$). 
\end{itemize}

In several places we make use of the function \texttt{StandardAffineShift}
defined in~\ref{SAS} to describe pseudo-random field elements or
polynomials.

\subsection{Case $r = p$} \label{secartinschreier}

In this case we use Artin-Schreier polynomials.

\begin{Pro}\label{defA} Let
\[
\begin{array}{rcl}
f_{p,1} & := & X_{p,1}^p - X_{p,1} - 1\\
f_{p,i} & := & X_{p,i}^p - X_{p,i} - (\prod_{j=1}^{i-1} X_{p,j})^{p-1}
\textrm{ for } i \geq 2.
\end{array}
\]
For each $l \in \Z_{>0}$ the  sequence $((X_{p,i}, f_{p,i}), 1 \leq i \leq
l)$ defines a tower of field extensions of degree $p$ over $\GFp$.
\end{Pro}

\begin{proof}
We use Lemma~\ref{degrirred}(a) and induction.  It is clear that $f_{p,1}$
has no zero in $\GFp$, so $f_{p,1}$ is irreducible over $\GFp$.

We write $I_i$  for the ideal generated by $\{f_{p,j}\mid\;  1 \leq j \leq
i\}$  and  $x_{p,i}$ for  the  residue  class of  $X_{p,i}$  in  $ F_i  :=
\GFp[X_{p,j}\mid\; 1 \leq j \leq i] / I_i$.

Now assume  $i > 1$ and  that we have shown  for $1 \leq j  \leq i-1$ that
$f_{p,j}$ modulo $I_{j-1}$ is irreducible  over $F_{j-1}$. We need to show
that $f_{p,i}$ modulo $I_{i-1}$ has no zero in $F_{i-1}$.

Set $y := x_{p,i-1}$ and  $a := (\prod_{j=1}^{i-2}x_{p,j})^{p-1}$, so that
$y^p -  y -  a =  0$ and  $f_{p,i} \mod  I_{i-1} =  X_{p,i}^p -  X_{p,i} -
ay^{p-1}$. Each element of $F_{i-1}$ has the form
\[ x = c_0 + c_1 y + \ldots + c_{p-1} y^{p-1}, \quad 
\textrm{ with unique }c_k \in F_{i-2}. \]

We evaluate $f_{p,i} \mod I_{i-1}$ at $x$,
\[ z = x^p - x - ay^{p-1} \in F_{i-1}, \]
and use  $y^p = y  + a$ to  see that the  coefficient of $y^{p-1}$  in $z$
is  $c_{p-1}^p  - c_{p-1}  -  a$.  So $z  \neq  0$  because by  assumption
$X_{p,i-1}^p  -  X_{p,i-1}  -a  \in F_{i-2}[X_{p,i-1}]$  has  no  zero  in
$F_{i-2}$.
\end{proof}

\subsection{Case $r \mid (p-1)$ and $4 \mid (p-1)$ if $r=2$}
\label{rinpminus1}

The assumption means that $\GFp^\times$ contains primitive $r$-th roots of
unity.  Equivalently,  $\GFp^\times$  contains  an element  $a$  which  is
no  $r$-th  power. (Note  that  $x  \mapsto  x^r$  is an  automorphism  of
$\GFp^\times$ if $r  \nmid (p-1)$.) Once we have specified  such an $a \in
\GFp^\times$ it  is again very easy  to define polynomials we  are looking
for.

An  element  $a  \in \GFp^\times$  is  an  $r$-th  power  if and  only  if
$a^{(p-1)/r} = 1$.

\begin{Alg}[non $r$-th power]\label{nonr}
Input: $F$, $r$,

where $F$  is a  finite field  whose elements  are identified  by Steinitz
numbers, and $r$ is a prime number dividing $|F^\times|$.

Output: An element $a \in F$ that is not an $r$-th power in $F$.

\begin{itemize}
\item[(a)] Initialize $i = 0$, $a = 0 \in F$.
\item[(b)] While $a = 0$ or $a^{(|F|-1)/r} = 1$ do:

\mbox{}\quad\quad $i = i+1$ and set $a \in F$ to the element with Steinitz
number\\
\mbox{}\quad\quad \texttt{StandardAffineShift}$(|F|, i)$ (see~\ref{SAS}).
\item[(c)] Return $a$.
\end{itemize}

\end{Alg}

\noindent\textbf{Proof     and      remark.}     Since      $i     \mapsto
\texttt{StandardAffineShift}(|F|, i)$ is a  bijection on the integers from
$0$ to $|F|-1$ the element $a$ will run through all elements of $F$ and so
the algorithm will find an element that is no $r$-th power.

In fact  the algorithm will  finish very  quickly in practice  because the
proportion of $r$-th powers in $F^\times$  is only $1/r$, and the order in
which we  run through $F$  looks like a  random order. In  experiments the
performance  was not  different  from using  a  more sophisticated  random
number generator instead of \texttt{StandardAffineShift}.

When $F$  is a prime  field, then running  through $F$ by  Steinitz number
would also work well, but for  non-prime fields the small Steinitz numbers
all refer  to elements in a  proper subfield and  then it can take  a long
time to find the first element which is not an $r$-th power.
{\mbox{}\hfill $\Box$} \medskip 

\begin{Pro}\label{defB}
For given primes $p$  and $r$ with $r \mid (p-1)$ let $a  \in \GFp$ be the
element that  is not  an $r$-th power  found by  Algorithm~\ref{nonr} with
inputs $\GFp$ and $r$.

We define polynomials in $\GFp[X_{r,i} \mid\: i \in \mathbb{Z}_{>0}]$:
\[
\begin{array}{rcl}
f_{r,1} & := & X_{r,1}^r - a\\
f_{r,i} & := & X_{r,i}^r - X_{r,i-1}
\textrm{ for } i \geq 2.
\end{array}
\]
For each $l \in \Z_{>0}$ the  sequence $((X_{r,i}, f_{r,i}), 1 \leq i \leq
l)$ defines a tower of field extensions of degree $r$ over $\GFp$.
\end{Pro}

\begin{proof}
First assume that $r$ is odd.

The polynomial $f_{r,1}$ has no root  in $\GFp$ by construction of $a$ and
so it is irreducible by Lemma~\ref{degrirred}(b).

Now let $I_i$ be the ideal generated by $f_{r,j}$ with $j \leq i$ and $F_i
= \GFp[X_{r,j}\mid\; j \leq i] /  I_i$. Let $x_{r,i}$ be the residue class
of  $X_{r,i}$  in $F_i$.  Let  $r^t  || (p-1)$,  so  the  $r$-part of  the
order  of $a$  is $r^t$.  By  construction the  $r$-part of  the order  of
$x_{r,i}$  is  $r^{t+i}$ for  $i  \geq  1$. And  by  Lemma~\ref{rparts}(b)
we  have  $r^{t+i} ||  \;  |F_i-1|$.  This  shows  by induction  that  all
polynomials $f_{r,i} \textrm{ modulo } I_{i-1} = X_{r,i}^r - x_{r,i-1} \in
F_{i-1}[X_{r,i}]$  have no  zero in  $F_{i-1}$ and  so are  irreducible by
Lemma~\ref{degrirred}(b).

In the case $r=2$ we assume $4 |  (p-1)$ and so $2 || (p+1)$. In this case
Lemma~\ref{rparts}(c) shows that  the statement in~\ref{rparts}(b) remains
correct for $r=2$. So, our proof also holds in this case.
\end{proof}

\subsection{Case $r=2$ and $4 | (p+1)$} \label{secr2}

In this case we have $2 || (p-1)$  and $-1 \in \GFp$ has no square root in
$\GFp$,  that  is  $X_{2,1}^2  +  1  \in  \GFp[X_{2,1}]$  is  irreducible.
We  construct  $\GF{p^2}$  as  extension of  $\GFp$  via  this  polynomial
and use the  corresponding Steinitz  numbering  of $\GF{p^2}$.  

In the following proposition let $a  \in \GF{p^2}$ be the element that has
no square root in $\GF{p^2}$  returned by Algorithm~\ref{nonr} with inputs
$\GF{p^2}$ and $2$.

\begin{Pro}\label{defC} Recall $4 | (p+1)$.
We define polynomials in $\GFp[X_{2,i} \mid\: i \in \mathbb{Z}_{>0}]$:
\[
\begin{array}{rcl}
f_{2,1} & := & X_{2,1}^2 + 1\\
f_{2,2} & := & X_{2,2}^2 - a\\
f_{2,i} & := & X_{2,i}^2 - X_{2,i-1}
\textrm{ for } i \geq 3.
\end{array}
\]
For each $l \in \Z_{>0}$ the  sequence $((X_{2,i}, f_{2,i}), 1 \leq i \leq
l)$ defines a tower of field extensions of degree $2$ over $\GFp$.
\end{Pro}

\begin{proof}
The   proof  is   similar   as  for   Proposition~\ref{defB},  now   using
Lemma~\ref{rparts}(c).
\end{proof}

\subsection{Case $r \neq p$, $r \nmid (p-1)$} \label{secvar2}

The idea  for this generic case  is simply to construct  relatively sparse
pseudo-random  polynomials  and  to  check them  for  irreduciblity.  From
Lemma~\ref{numirrpols} we know  that about $1/r$ of  all monic polynomials
of degree $r$ are irreducible.

In the  next algorithm  \texttt{FindIrreduciblePolynomial(K, r, a,  X)} we
assume that the argument \texttt{K} is a finite field which has a Steinitz
numbering.
The  argument   \texttt{r}  is  a  positive  integer,
\texttt{a}  is  a nonzero  element  of  \texttt{K}  and \texttt{X}  is  an
indeterminate  over  \texttt{K}.  The   function  returns  an  irreducible
monic  polynomial of  degree \texttt{r}  in the  variable \texttt{X}  over
\texttt{K}  with  constant term  \texttt{a}.  We  assume that  a  function
\texttt{IsIrreducible(K,  f)}  is  available  that  checks  if  the  monic
polynomial \texttt{f} over  \texttt{K} is irreducible. Here  is the pseudo
code:

\begin{Alg}\label{randirrpol}
\mbox{}\\
\texttt{FindIrreduciblePolynomial(K, r, a, X)}\\
\mbox{}\hspace{2ex}\texttt{q = |K|}\\
\mbox{}\hspace{2ex}\texttt{inc = minimal integer 
with $\texttt{q}^\texttt{inc} \geq 2 \texttt{r}$}\\
\mbox{}\hspace{2ex}\texttt{d = 0 \quad \textrm{(random coeffs 
up to \texttt{X}$^\texttt{d}$)}}\\
\mbox{}\hspace{2ex}\texttt{f = X$^\texttt{r}$ + X + a \textrm{(first 
polynomial to try)}}\\
\mbox{}\hspace{2ex}\texttt{count = 0}\\
\mbox{}\hspace{2ex}\texttt{while not IsIrreducible(K, f)}\\
\mbox{}\hspace{6ex}\texttt{if count modulo r = 0 then}\\
\mbox{}\hspace{10ex}(after any \texttt{r} trials we allow
\texttt{inc} more non-zero coefficients)\\
\mbox{}\hspace{10ex}\texttt{d = minimum(d+inc, r-1)}\\
\mbox{}\hspace{6ex}\texttt{s = StandardAffineShift(q$^\texttt{d-1}$, count)} 
(see~\ref{SAS})\\
\mbox{}\hspace{6ex}Let \texttt{g $\in$ K[X]} be the polynomial with Steinitz
number \texttt{s} (see~\ref{DefSteinitz}), set\\
\mbox{}\hspace{6ex}\texttt{f = X$^\texttt{r}$ + g X + a}\\
\mbox{}\hspace{6ex}\texttt{count = count + 1}\\
\mbox{}\hspace{2ex}\texttt{return f}\\
\end{Alg}

\noindent\textbf{Proof and remark.} 
The  correctness of  the  algorithm  is clear,  since  we will  eventually
run  through  all monic  polynomials  of  degree  $r$  and test  them  for
irreducibility.

The proportion of monic polynomials of degree $r$ with prescribed constant
term which is irreducible is about $1/r$ by Lemma~\ref{numirrpols}(b).

In practice,  running through the  polynomials in  the order given  by the
function \texttt{StandardAffineShift}  shows the same performance  as with
using any sophisticated random number generator.

It  is not  advisable  to run  through the  polynomials  just by  Steinitz
number.  We have  tried  this  and occasionally  found  examples where  no
irreducible polynomial was found after  very long running times. (Example:
We  tried  a  huge  number   of  polynomials  of  form  $X^{107}+bX+a  \in
\GF{2^{107}}[X]$ without finding any irreducible one. With our strategy in
\texttt{FindIrreduciblePolynomial}  we  only  try $107$  such  polynomials
first and from  then allow a non-zero  coefficient of $X^2$ and  so on. In
this case, our  irreducible polynomial has non-zero  coefficients also for
$X^3$ and $X^4$.)
{\mbox{}\hfill $\Box$} \medskip 

We sketch  a practical way  to check whether a polynomial $f \in  K[X]$ of
degree $r$, where $|K|=q$, is  irreducible. The polynomial $f$ contains an
irreducible factor of degree dividing $t$  if and only if $\gcd(f, X^{q^t}
- X)$ has positive degree. We have $\gcd(f, X^{q^t} - X) = \gcd(f, h - X)$
where  $h \equiv  X^{q^t} \mod  f$ can  be computed  by repeated  squaring
modulo $f$. So, $f$ is irreducible if  and only if $\gcd(f, X^{q^t} - X) =
1$  for $1  \leq  t  \leq r/2$.  Many  non-irreducible random  polynomials
contain a factor of small degree  which is quickly detected by this method
(see comments on Ben-Or's test in~\cite{GaoPan}).

For a speedup we precompute $(X^0)^q, X^q, \ldots, (X^{r-1})^q \mod f$ and
use that  $x\mapsto x^q$ is a  $K$-linear map to compute  $X^{q^j} \mod f$
for $j > 1$.

Now we define a tower of field  extensions of degree $r$ over $\GFp$.

\begin{Def}\label{defE}
Set   $f_{r,1}   =   $\texttt{FindIrreduciblePolynomial}($\GFp,   r,   -1,
X_{r,1}$).

Assume that a tower of field extensions of degree $r$,
\[ ((X_{r,1}, f_{r,1}), \ldots, (X_{r,i-1}, f_{r, i-1})) \]
is already defined and set $F_{i-1}  = \GFp[X_{r,j}\mid 1\leq j \leq i-1]/
(f_{r,1},  \ldots,  f_{r,i-1})$, and  write  $x_{r,i-1}$  for the  residue
class  of $X_{r,i-1}$  in  $F_{i-1}$.  

Then compute the polynomial 
\[ \texttt{FindIrreduciblePolynomial}(F_{i-1}, r, -x_{r,i-1}, X_{r,i}) \]
and  substitute the  $x_{r,j}$ with  $j <  i$ in  the coefficients  by the
representing variables $X_{r,j}$ to define $f_{r,i}$.

Note that the norm  of $x_{r,1}$ over $\GFp$ is $1 \in \GFp$  and for $i >
1$ the norm of $x_{r,i}$ over $F_{i-1}$ is $x_{r,i-1}$.
\end{Def}

\textbf{Remarks.} We have tried various methods for generating irreducible
polynomials described  in the literature. But  we did not find  any method
that  worked as  well  as  testing random  polynomials  in  general. In  a
previous version  of this paper  we had a  variant that was  described and
analysed  by  Shoup  and  is  cited  in  many  articles,  see~\cite{Shoup,
ShoupFast}. If $r \nmid (p-1)$ the idea is to first construct an extension
$\GF{p^e}$ that  does contain elements  of order  $r$, then procede  as in
case $r | (p-1)$ above, and finally  use the traces of the generators into
the fields  of order  $p^{r^i}$ as  generators of  the prime  power degree
extensions. Unfortunately, the intermediate degree  $e$ can be as large as
$r-1$ and  then one needs  to compute temporarily  in a much  bigger field
than one  wants to construct. Despite  some efforts we could  not get this
method sufficiently  efficient in  practice. (And,  of course,  the method
in~\ref{secvar2} is much easier to describe and implement.)

On  the other  hand, why  don't we  simplify our  description further  and
use~\ref{secvar2} for all $r$? Here, in  the special cases it is very easy
to just right down  polynomials we want, and for $r =  2$ always and other
small $r$ often  one of the special cases applies.  In our practical tests
the  special  cases  yield  a  noticable  speedup  compared  to  searching
pseudo-random irreducible polynomials for all $r$.

\section{Definition of standard generators of cyclic subgroups}
\label{secdefcyc}

In  this  section  we  define  explicit  generators  $y_{n,r}$  of  cyclic
subgroups  of order  $r^t$ where  $r$ is  prime with  $r^t ||  (p^n-1)$ as
described in Remark~\ref{defcyclic}.

We describe the elements $y_{n,r}$ as output of an algorithm.

The  construction  is  relative   to  some  standardized  construction  of
$\bar\GFp$ where we can identify each  element in any finite subfield by a
Steinitz number.

In the base case of the following  algorithm we use again the numbering of
elements defined by \texttt{StandardAffineShift}, see~\ref{SAS}.

\begin{Alg} \texttt{StandardCyclicGeneratorPrimePower}$(p, n, r)$
\label{standardcyc}

Input: a prime $p$, a degree $n$ and a prime $r$ with $r | (p^n-1)$.

Output: an element $y_{n,r} \in \GF{p^n}^\times$ of order $r^t$ where $r^t
|| (p^n-1)$.

\begin{itemize}
\item[(a)] Find $t$ and minimal divisor $k | n$ such that $r^t || (p^k-1)$.
If $k < n$  then return the result of 
\texttt{StandardCyclicGeneratorPrimePower}$(p, k, r)$ as element of
$\GF{p^n}$.
\item[(b)] (Find base case)
\begin{itemize}
\item[(b1)] If $r = 2$,  $p \equiv 3 \textrm{ mod } 4$ and $2\mid n$
then set $l = 2$.
\item[(b2)] Otherwise find minimal $l|n$ with $r \mid (p^l-1)$.
\end{itemize}

\item[(c)] Base case $l = n$: 

\begin{itemize}
\item[(c1)]
Initialize \texttt{count = 0}, $x = 0 \in \GF{p^n}$
\item[(c2)]
While $x = 0$ or $x^{(p^n-1)/r} = 1$ do\\
\mbox{}$\quad\quad$ \texttt{count = count + 1}\\
\mbox{}$\quad\quad$ \texttt{s = StandardAffineShift($p^n$, count)}\\
\mbox{}$\quad\quad$ $x = $ element in $\GF{p^n}$ with Steinitz number 
\texttt{s}\\
\item[(c3)]
Return $y_{n,r} = x^{(p^n-1)/r^t}$.
\end{itemize}

\item[(d)] Case $l < n$: 

We need  to find  a generator  which is  compatible with  the choice  in a
proper subfield.

In this case $r \mid n$ and $r^{t-1} || (p^{(n/r)}-1)$, see~\ref{rparts}.

Return the Steinitz-smallest $r$-th root of $y_{n/r, r}$.
\end{itemize}
\end{Alg}

\begin{Rem}\label{remcyc}
Note that for  $r = 2$ and $p \equiv  3 \textrm{ mod } 4$ in  case $n = 1$
step~(c) will always  return $-1 \in  \GFp$;  and for $n =  2$ any element
found in the base  case will be automatically compatible with  the $n = 1$
case.

The time critical  case in this algorithm is step~(d).  A practical method
for this  step is  to first power  up random elements  to find  an element
$y$  of  order  $r^t$.  Then  compute  the  discrete  logarithm  $b$  such
that  $(y^r)^b =  y_{n/r,  r}$. This  can be  done  by the  Pohlig-Hellman
algorithm~\cite{PoHe} which  involves $t-2$ searches through  $r$ elements
(which can  be optimized via  Shanks' algorithm).  Now $y^b$ is  an $r$-th
root of $y_{n/r, r}$. Finally multiply  $y^b$ with all $r$-th roots of one
(these are  the elements $(y^{r^{t-1}})^i$,  $0 \leq i  < r$) to  find the
$r$-th root  of $y_{n/r,r}$ with  the smallest Steinitz number.  Note that
the primes $r$  for which this case  occurs are divisors of  the degree of
the field over its prime field, and so loops of length $r$ are acceptable.

Similar  to  former  remarks  we  have  again  noticed  that  the  use  of
\texttt{Standard\-Affine\-Shift} in the base case~(c) behaves similar to a
randon choice of elements to  try. More systematic strategies, like trying
elements by ascending or descending Steinitz number lead to cases where no
result was returned after long running times.

\end{Rem}

\section{Remarks on implementation}

The  main  purpose   of  this  article  is  to   describe  a  standardized
construction of finite fields and  standardized generators of their cyclic
subgroups which works in practice and could be adopted by various software
packages dealing with finite fields.

Therefore,   we  publish   at   the   same  time   as   this  article   an
implemention  of  our constructions  as  a  GAP~\cite{GAP} package  called
StandardFF~\cite{StandardFF}.

\subsection{Implementation  of  the  standardized  extensions  in  other
programs}

One could implement  $\bar\GFp$, or its subfield   $\GF{p^n}$, directly as
explained in Section~\ref{secbarFp}  and Remark~\ref{remcomputetower} such
that the elements are represented by multivariate polynomials in variables
$X_{r,i}$ for primes $r$ with $r^i | n$. But this is not an efficient
representation for computations.

Instead  we construct  $\GF{p^n}$ as  simple extension  $\GFp[x_n]$, where
$x_n$ is the primitive element that we have defined in~\ref{rembargfp}(e),
together  with an  $(n \times  n)$-matrix  whose $i$-th  row contains  the
coefficients of  $x_n^{i-1}$ expressed in  the tower basis  of $\GF{p^n}$.
(The inverse of this transition matrix expresses the elements of the tower
basis as linear combination of the basis $(1, x_n, \ldots, x_n^{n-1})$.)

We compute  this recursively. In the  case $n=1$ we represent  $\GFp \cong
\mathbb{Z}/p\mathbb{Z}$  and $(1)$  is  the natural  basis  and the  tower
basis. Let $n  > 1$, $r$ the largest  prime divisor of $n$, $m  = n/r$ and
$r^t  n'=  n$ with  $\gcd(n',r)  =  1$. We  assume  that  we have  already
constructed  $\GF{p^m} =  \GFp[x_m]$ together  with the  transition matrix
from  the tower  basis to  the basis  $(1, x_m,  \ldots, x_m^{m-1})$.  The
generator  $x_{n'}$  of the  field  $\GF{p^{n'}}$  is  an element  of  the
tower basis  of $\GF{p^m}$.  Let $f_{r,t}(X_{r,t})$  be the  polynomial of
degree $r$  defined in  Section~\ref{secstdgf}. We  can write  the residue
classes  of  the coefficients  of  $f_{r,t}$  as elements  of  $\GF{p^m}$,
using  the  transition matrix  from  the  power  basis of  $\GF{p^m}$  and
get  $\bar{f}_{r,t}(X_{r,t})  \in  \GF{p^m}[X_{r,t}]$.  Now  consider  the
field  $\GF{p^m}[X_{r,t}]/(\bar{f}_{r,t}) \cong  \GF{p^n}$. Our  primitive
element  is $x_n  =  x_{n'}  x_{r,t}$ (as  in  earlier  sections we  write
$x_{r,t}$ for  the residue class of  $X_{r,t}$). If $(b_0,\ldots,b_{m-1})$
is   the   tower  basis   of   $\GF{p^m}$   then  the   concatenation   of
$(b_0x_{r,t}^i,\ldots,b_{m-1}x_{r,t}^i)$ for $i = 0, 1,\ldots, r-1$ is our
tower basis of $\GF{p^n}$.

It is now straight forward to  express the elements $1, x_n, x_n^2, \ldots
x_n^{n-1}$ in this tower basis  (whenever multiplication by $x_n$ leads to
a term containing  $x_{r,t}^r$ we substitute this by  a linear combination
of lower powers of $x_{r,t}$ using $\bar{f}_{r,t}$ (here it is useful that
our $\bar{f}_{r,t}$ are often sparse).

The  next power  $x_n^n$  can be  written als  linear  combination of  the
previous ones and  this yields the minimal polynomial  $f_n(X_n)$ of $x_n$
in  $\GFp[X_n]$. So,  $\GF{p^n} =  \GFp[X_n]/(f_n(X_n))$ and  we have  the
transisition matrix from the powers of $x_n$ to the tower basis.

Our software  package supports  various representations  of elements  in a
field  $\GF{p^n}  = \GFp[X_n]/(f_n(X_n))$:  as  polynomials  in $X_n$,  as
coeffient  vectors  with  respect  to the  tower  basis,  as  multivariate
polynomials  as in  Remark~\ref{rembargfp}(a), as  Steinitz numbers  or as
Steinitz pairs, see~\ref{rembargfp}(c). There are functions to compute our
standardized generators of cyclic subgroups  and embeddings of fields. The
arithmetic  of elements  in different  fields is  also supported  by first
mapping the operands into a common larger field.

Our  implementation   only  uses  arithmetic  of   univariate  polynomials
(represented as coefficient lists) and  contains an irreducibility test as
mentioned  after~\ref{randirrpol}. We  compute  minimal  polynomials of  a
field  element by  computing  its action  on some  basis  and the  minimal
polynomial of the corresponding matrix.

In further systematic tests we considered the finite fields of order $p^n$
in the following ranges:

\begin{itemize}
\item $1 \leq n \leq 2000$ for $p = 2,3,5,7$
\item $1 \leq n \leq 500$ for $10 < p < 100$
\item $1 \leq n \leq 100$ for $100 < p < 10000$
\end{itemize}

This includes all 10800 cases for which we know the Conway polynomial. Due
to  decades long  (and  ongoing) enormous  computational  efforts to  find
factors of  numbers of the  form $a^n \pm 1$,  see~\cite{BrentFactors}, we
know  the factorization  of $p^n-1$  for 112968  fields in  the considered
range (May 2022). These are the only  fields for which we can hope to find
(standardized) primitive roots (otherwise we cannot determine the order of
an element in the field).

Our programs can construct all of  these 112968 fields $\GF{p^n}$ in about
7 hours  and it can find  all the standardized primitive  roots $y_{p^n-1}
\in  \GF{p^n}$ as  described  in Remark~\ref{defcyclic}  in additional  21
hours. The minimal polynomials of these $y_{p^n-1}$ over their prime field
form a substitute  for the Conway polynomials with  the same compatibility
properties. Computing  the $y_{p^n-1}$ and their  minimal polynomials just
for  the fields  where we  know the  Conway polynomial  takes less  than 2
minutes (while the  original computations of the  known Conway polynomials
involved many years of CPU time).

It is also  possible to construct many fields outside  the mentioned range
(larger degree or  much larger characteristic). The hard cases  are when a
large  prime  divides  the  degree.  

\subsection{Computing embeddings}

Embeddings    are     easily    computed    via    the     tower    bases,
see~\ref{DefTowerBasis}. The  ordered tower  basis of  $\GF{p^n}$ contains
the tower basis of each subfield as subsequence. The list of degrees (over
the prime field) of the tower basis elements can be generated as follows:

If  $n =  1$ it  is $(1)$.  For $1  < n  = r_1^{l_1}\cdots  r_k^{l_k}$ let
$(d'_1, \ldots, d'_{n/r_k})$ be the list  of degrees of the tower basis of
$\GF{p^{n/r_k}}$   (these are  the  first $n/r_k$  elements  of the  tower
basis  of $\GF{p^n}$).  Then we  get the  degrees for  the tower  basis of
$\GF{p^n}$  by appending  $(r-1)$ times  $(\lcm(d'_1, r_k^{l_k}),  \ldots,
\lcm(d'_{n/r_k}, r_k^{l_k}))$.

Let $(b_1,  \ldots, b_n)$ be  the tower  basis of $\GF{p^n}$  with degrees
$(d_1,\ldots,  d_n)$. Let  $m \mid  n$. Then  the subsequence  $(b_j\mid\;
d_j|m)$ is the tower basis of $\GF{p^m}$.

Let $x =  \sum_{i=1}^n a_i b_i \in \GF{p^n}$ ($a_i  \in \GFp)$, written in
the tower  basis. Then the degree  of $x$ over $\GFp$  is $\lcm\{d_j\mid\;
a_j \neq 0\}$.

\section{Application to Brauer character tables}\label{sec:brauer}

Let $G$ be a finite group, $K$ be an algebraically closed field and $n \in
\Z_{>0}$. A group homomorphism $\rho:  G \to \textrm{GL}_n(K)$ is called a
representation.  If  $K$  has  characteristic $0$  then  much  information
about  $\rho$ is  encoded in  its character  $\chi: G  \to K$,  $g \mapsto
\textrm{Trace}(\rho(g))$,  a  function  which  is  constant  on  conjugacy
classes. The $\textrm{Trace}$ is the sum  of the eigenvalues of the matrix
in $K$.

If $K$ has  finite characteristic $p$ then the character  as defined above
contains much less information about $\rho$ (for example in characteristic
$0$ we have $\chi(1) = n$ but in characteristic $p$ only $\chi(1) = n \mod
p \in \GFp$). Instead we use \emph{Brauer characters} $\tilde{\chi}$ which
are defined on the elements $g \in G$ of order not divisible by $p$. Here,
$\tilde{\chi}(g)$  is the  sum of  the \emph{lift}  of the  eigenvalues of
$\rho(g)$ to  complex roots  of unity.  Such a lift  is an  isomorphism $e
\circ \ell: \bar\GFp^\times  \to \mu_{p'} \subseteq \C^\times$  as we have
considered  in  Section~\ref{secembedding}. In  general  the  values of  a
Brauer character depend on the chosen lift.

A large collection of Brauer characters is contained in the GAP~\cite{GAP}
character table  library CTblLib~\cite{CTblLib} which includes  all Brauer
characters from the Modular Atlas~\cite{ModAtl}. The values are given with
respect to a lift defined by the Conway polynomials: if $f(X) \in \GFp[X]$
is the Conway  polynomial defining the field with $p^n$  elements then the
lift  restricted to  $\GF{p^n}  \cong \GFp[X]/(f)$  is  defined by  $X+(f)
\mapsto  \exp(2 \pi  i  / (p^n-1))$.  More  details can  be  found in  the
Introduction and Appendix~1 of the Modular Atlas~\cite{ModAtl}.

Our definition of standard generators $y_m$  of cyclic groups of order $m$
in~\ref{defcyclic} and~\ref{standardcyc} yields  another lift, where $y_m$
is mapped to $\exp(2 \pi i / m)$ for all $m \in \Z_{>0}$.

How can we recompute the values of known Brauer characters with respect to
the  new lift  defined  here? The  first  step is  to  compute the  lifted
eigenvalues  from the  character  tables.  This can  be  done because  the
mentioned Brauer character tables contain the power maps of the group (for
each element  one knows the  conjugacy class of  all its powers)  and this
yields a  Vandermonde type system  of equations for the  multiplicities of
the eigenvalues. If the relevant Conway  polynomial is known we can derive
the corresponding  eigenvalues in  characteristic $p$  as elements  in the
finite fields defined by Conway polynomials.

The missing step  to compute the image of these  eigenvalues under our new
lift is an  identification of the elements in  Conway polynomial generated
fields  with  elements in  the  algebraic  closure $\bar\GFp$  constructed
in~\ref{secbarFp}.

Let  $n|m$ and  $f,g \in  \GFp[X]$ be  the Conway  polynomials of  degrees
$n$   and  $m$,   respectively.  Then   $\GFp[X]/(f)$  is   considered  as
subfield  of $\GFp[X]/(g)$  by mapping  $X+(f)$ to  $(X+(g))^a$ with  $a =
(p^m-1)/(p^n-1)$. We  define embeddings  of the fields  $\GFp[X]/(f)$ into
our $\bar\GFp$ which commute with these inclusion relations:

\begin{Def} \label{identConway}
We define the embedding by induction over the degree of the field.

Since  any field  homomorphism  maps $1  \mapsto  1$ it  is  clear how  to
identify the zero $z_1$ of a Conway  polynomial of degree $1$ in the prime
field. Now  let $n > 1$  and $f \in  \GFp[X]$ be the Conway  polynomial of
degree $n$.  Assume that for  any proper divisor $m  | n$ we  have already
defined the image  $z_m \in \bar\GFp$ of  the residue class $X  + (g)$ for
the Conway polynomial $g$ of degree $m$.

Then we map  $X + (f)$ to the  zero $z_n$ of $f$ in our  standard field of
order $p^n$  which has the  smallest Steinitz  number among the  zeros $z$
which fulfill the compatibility conditions $z^{(p^n-1)/(p^m-1)} = z_m$ for
all proper divisors $m$ of $n$.
\end{Def}

Our  software  package  StandardFF~\cite{StandardFF} contains  a  function
\texttt{Steinitz\-Pair\-Conway\-Generator}  which  computes  the  Steinitz
pairs decribing the elements $z_n$ in Definition~\ref{identConway}.

To compute  the image  of $z_n$  under our new  lift we  have to  find the
discrete  logarithm  $e$  such  that  $y_{p^n-1}^e =  z_n$.  This  can  be
challenging in  large fields,  but in  practice we  usually only  need the
image of powers of $z_n$ of small order which can be found much easier.

Our   package    StandardFF~\cite{StandardFF}   also   has    a   function
\texttt{Standard\-Values\-Brauer\-Character}  which  recomputes values  of
Brauer  characters with  respect to  our new  lift, provided  the relevant
Conway polynomials are known.

We  consider two  explicit examples:  The  Brauer character  table of  the
largest  sporadic  simple  group,  the  Monster,  in  characteristic  $19$
contains several  Brauer characters  for which  our function  to recompute
their values according  to our new lift fails (because  some needed Conway
polynomials are not  known and essentially impossible to  compute). But in
this case one can check that with any irreducible Brauer character all its
Galois conjugate  class functions are also  irreducible Brauer characters.
In such a case  the Brauer characters are the same for  any lift, only the
map  from a  set of  concrete  representations to  their Brauer  character
depends on the lift (and here  we cannot compute the permutation of Brauer
characters caused by the different lifts).

Thomas    Breuer   systematically    determined   all    cases   in    the
CTblLib~\cite{CTblLib}  library  where  we  cannot  recompute  the  Brauer
character values for  our new lift because of  missing Conway polynomials,
or  where a  complex  character value  cannot be  reduced  modulo a  prime
dividing the group order. This concerns about $50$ finite fields for which
the Conway polynomial would be  needed. With the construction described in
this paper  our software only  needs 3  seconds to construct  these fields
including our standardized primitive roots.

As second  explicit example we mention  the Brauer character table  of the
alternating group $A_{18}$ in characteristic $3$. In this case we are able
to recompute all  values with respect to  our new lift. It  turns out that
for each of the degrees $6435$  and $73645$ there are two characters where
the  new lift  yields  class  functions which  are  not  contained in  the
original table.  So, the  two different lifts  actually lead  to different
Brauer character tables.

Finally,  we  want  to  illustrate  another  interesting  feature  of  the
constructions in this  paper. Say, we have a group  element of order $523$
and we want  to lift $523$-th roots of unity  in characteristic $13$. Then
the smallest field containing such  roots of unity is $\GF{13^{261}}$. The
factorization of  $13^{261}-1$ is not known  and so probably very  hard to
compute. So, even with our new definition we have no chance to compute our
standardized primitive root  of this field. Nevertheless,  for our purpose
we only need  to construct the field of order  $13^{261}$ and our standard
generator of order  $523$ in this field.  Our programs can do  this in 0.4
seconds.

\sloppy
\printbibliography
\end{document}